\documentclass[12pt, reqno]{amsart}
\usepackage{amsmath, amsthm, amscd, amsfonts, amssymb, graphicx }
\usepackage[bookmarksnumbered, colorlinks, plainpages]{hyperref}
\hypersetup{colorlinks=true,linkcolor=red, anchorcolor=green, citecolor=cyan, urlcolor=red, filecolor=magenta, pdftoolbar=true}

\newtheorem{theorem}{Theorem}[section]
\newtheorem{lemma}[theorem]{Lemma}
\newtheorem{proposition}[theorem]{Proposition}
\newtheorem{corollary}[theorem]{Corollary}
\theoremstyle{definition}

\newtheorem{example}[theorem]{Example}

\theoremstyle{remark}

\numberwithin{equation}{section}

\begin{document}

\setcounter{page}{1}
\title[ Fuglede - Putnam Theorem  ]{ Asymmetric Fuglede - Putnam theorem for Unbounded $M$-Hyponormal operators }
\author[ T. PRASAD, E. SHINE LAL AND P. RAMYA  ]{ T. PRASAD, E. SHINE LAL AND P. RAMYA }
\address{T. Prasad \endgraf Department of Mathematics\endgraf University of Calicut\endgraf Malapuram, Kerala, \endgraf India - 673635.}
\email{prasadvalapil@gmail.com }
\address{E. Shine Lal \endgraf Department of Mathematics\endgraf University College, Thiruvananthapuram\endgraf Kerala, India -695034.}
\email{shinelal.e@gmail.com }
\address{P. Ramya \endgraf Department of Mathematics\endgraf N.S.S College, Nemmara\endgraf
 Kerala, India -678508.}
\email{ramyagcc@gmail.com }

\dedicatory{ }

\let\thefootnote\relax\footnote{}

\subjclass[2010]{47A05, 47A10, 47B20}

\keywords{Closed densely defined $M$-hyponormal operator, Fuglede-Putnam Theorem, Riesz projection.}

\begin{abstract}
A closed densely defined operator $ T $ on a Hilbert space $ \mathcal{H} $ is callled $M$-hyponormal if
  $\mathcal{D}(T) \subset \mathcal{D}(T^{*}) $ and there exists $ M > 0 $ for which  $ \parallel(T-zI)^{*}x \parallel \leq M \parallel(T-zI)x \parallel $ for all $ z \in \mathbb{C}$
  and for all $ x\in \mathcal{D}(T)$. In this paper, we prove that if   bounded linear operator $ A : \mathcal{H} \rightarrow  \mathcal{K}$ is such that $ AB^*\subseteq TA $, where  $ B $ is a closed  subnormal (resp. a closed $ M $-hyponormal) on $\mathcal{H}$, $ T $ is a closed $ M $-hyponormal (resp. a closed subnormal) on $\mathcal{H}$, then (i) $ AB\subseteq T^*A, $ (ii) $ {\overline{ran(A^{*})}} $ reduces $ B $ to the normal operator $ B\vert_{{\overline{ran(A^{*})}}}, $  and (iii) $ {\overline{ran(A)}} $ reduces $ T $ to the normal operator $ T\vert_{{\overline{ran(A)}}}. $ 

\end{abstract}\maketitle
\section{Introduction and Preliminaries}
Let $ \mathcal{H},\mathcal{K},\mathcal{L} $ be infinite dimensional complex Hilbert space. Let $ \mathcal{B}(\mathcal{H}) $ denotes the space of all bounded linear operators on $ \mathcal{H} $. The famous \textit {Fuglede - Putnam theorem} asserts that if $ T,N \in \mathcal{B}(\mathcal{H})  $ are normal operators and $ TX=XN $ for some $ X \in \mathcal{B}(\mathcal{H}) $, then $ T^*X=XN^* $ (\cite{ Fuglede, Halmos, mor, putnam1}). Several authors tried to relax the normality condition (\cite{cho1, bach,  mech, moore, raja}). But for subnormal operators, it does not hold. Asymmetric version of Fuglede - Putnam theorem for bounded subnormal operators, $ M $-hyponormal operators and dominant operators are studied by Furuta (\cite{furuta, furuta1})  Moore (\cite{moore}) and  Stampfli (\cite{stamp}) respectively. The study of unbounded operators is one of the major research area in the field of operator theory. In most of the cases operators involved in quantum mechanics, differential equations are unbounded.

 Let $ \mathcal{D}(T),~ker(T),~ ran(T) $ and $\mathcal{G}(T) $ denotes the domain, nullspace, range and graph of $ T $ respectively. Let $A= U ~\vert A \vert$ be the polar decomposition of $ A \in \mathcal{B}(\mathcal{H}) $, where $\vert A \vert = (A^{*}A)^{\frac{1}{2}}$.
 Let $ \mathcal{L}(\mathcal{H})$ and $ \mathcal{C}(\mathcal{H})$ denotes the space of all linear and closed linear operators  on $\mathcal{H}$ respectively. Let  $ T\vert_{\mathcal{M}} $ denotes the restriction of $ T $ in to the closed subspace $\mathcal{M} $ of $ D(T)$. A closed subspace $\mathcal{M}$ is said to be a core for $ T $ if $ \mathcal{G}(T)\subseteq \overline {\mathcal{G}(T|_{\mathcal{M}})} $ (\cite{stochel}).

  An operator $ T \in \mathcal{L}(\mathcal{H}) $ is said to be \textit{densely defined} if $\overline{\mathcal{D}(T)} = \mathcal{H}.$ For example,  $ T:l^{2}(\mathbb{N})\longrightarrow l^{2}(\mathbb{N}) $
 be defined by
\begin{align*}
T(x_{1},x_{2},x_{3}......) = (2x_{1},3x_{2},4x_{3},5x_{4},6x_{5}......)
\end{align*}
with $\mathcal{D}(T) = \lbrace (x_{1},x_{2}........)\in l^{2}( \mathbb{N}):\sum^{\infty}_{j=1}\vert (j+1) x_{j}\vert^{2}< \infty \rbrace$.
Since $ C_{00}\subseteq \mathcal{D}(T) $, and $ C_{00} $ is dense in $ l^{2}(\mathbb{N}) $, we have $ \mathcal{D}(T) $ is dense in $l^{2}(\mathbb{N})$.

An operator $ B \in \mathcal{C}(\mathcal{H})$ is said to be \textit{subnormal} if there is a Hilbert space $ \mathcal{K} \supseteq \mathcal{H} $ and a normal operator $ S $ in $ \mathcal{K} $ such that $ B \subseteq S$. An operator $ T \in \mathcal{B}(\mathcal{H})$ is said to be \textit{hyponormal} if $ T^{*}T \geq TT^{*}$ and is said to be  \textit{M-hyponormal}
 if $\parallel(T-zI)^{*}x \parallel \leq M \parallel(T-zI)x \parallel $ for all $ z \in \mathbb{C}$ and for all $ x\in \mathcal{H}$. It is evident that 
the following inclusion hold:
 $$ hyponormal \subset M-hyponormal.$$
 A  densely defined operator $ T\in \mathcal{C}(\mathcal{H})$ is said to be  $M$-hyponormal if
  $ \mathcal{D}(T) \subset \mathcal{D}(T^{*}) $ and there exists $ M > 0 $ for which  $ \parallel(T-zI)^{*}x \parallel \leq M \parallel(T-zI)x \parallel $ for all $ z \in \mathbb{C}$
  and for all $ x\in \mathcal{D}(T).$ \\
  Now we give an example for closed densely defined $ M $- hyponormal opearator, which is not hyponormal.
\begin{example}
   Let $ T:l^{2}(\mathbb{N})\longrightarrow l^{2}(\mathbb{N}) $ be defined by
\begin{align*}
T(x_{1},x_{2},x_{3}......) &= (0,x_{1},2x_{2},x_{3},4x_{4},5x_{5}......)\\&=(0,\alpha_{1}x_{1},\alpha_{2}x_{2},\alpha_{3}x_{3},\alpha_{4}x_{4}.........),
\end{align*}
where $ \alpha_{j} $ = $ \left\{ \begin{array}{rcl}1 & \mbox{if}& j = 1,3\\j & \mbox{if} & j = 2 $  $ \rm {and} $  $ j \geq 4\end{array}\right . $ \\
Let
 $\mathcal{D}(T) = \lbrace (x_{1},x_{2}........)\in l^{2}( \mathbb{N}):\sum^{\infty}_{j=1}\vert \alpha_{j} x_{j}\vert^{2}< \infty \rbrace$.  Since $ C_{00}\subseteq \mathcal{D}(T) $, and $ C_{00} $ is dense in $ l^{2}(\mathbb{N}) $, $ \mathcal{D}(T) $ is dense in $l^{2}(\mathbb{N})$. Since $ (\alpha_{n}) $ is eventually increasing, $ T $ is $M$-hyponormal (\cite{jun}). The adjoint of $ T, T^{*} $ is given by
\begin{align*}
  T^{*}( x_{1},x_{2},x_{3}......) =  (x_{2},2x_{3},3x_{4},4x_{5},5x_{6}......).
\end{align*}
Let $ e_{i} = (0,0,....,1,0,0,.....)$, where 1 occurs in the $ i^{th} $ place. Then
$$Te_{1} = e_{2},Te_{2} = 2e_{3}, Te_{3}=e_{4}, Te_{i}=ie_{i+1}  \mbox{ for } i  \geqslant 4.$$
 $$T^{*}e_{1}=0, T^{*}e_{2}=e_{1},T^{*}e_{3}=2e_{2},T^{*}e_{4}=e_{3},T^{*}e_{i}=(i-1)e_{i-1} \mbox { for }    i \geqslant 5.$$
Since $ \Vert T^{*}e_{3}\Vert = 2 $ and $ \Vert Te_{3}\Vert = 1, $ it follows that $ T $ is not hyponormal.
\end{example}
 Stochel (\cite{stochel}) studied asymmetric Fuglede - Putnam theorem for closed hyponormal (resp. closed subnormal) and closed subnormal (resp. closed hyponormal) operators. Recently, Bensaid, Dehimi, Fuglede and Mortad(\cite{BDFM}), studied new and classic version of Fuglede theorem in an unbounded setting. Assymmetric Fuglede - Putnam theorem  for some class unbounded operators has been studied by Mortad (\cite{mortad, mortad1}) and  Paliogiannis (\cite{palio}).  In this paper we study asymmetric  Fuglede - Putnam theorem for closed $ M $-hyponormal (resp. closed subnormal) and subnormal (resp. closed $ M $-hyponormal) operators by the method of (\cite{stochel}).
 
\section{ Fuglede-Putnam Theorem}
 A closed subspace $ \mathcal{M} $ of $ \mathcal{H} $ is said to be \textit{invariant} under $ T \in \mathcal{L}(\mathcal{H})$
if for any $ x \in \mathcal{D}(T) \cap \mathcal{M}$, then $ Tx \in \mathcal{M}$. ie., if $ \mathcal{M} $ is a closed subspace of $ \mathcal{H} $,  we define $ T\vert_{ \mathcal{M}} $ is an operator on $ \mathcal{M} $ with domain $$ \mathcal{D}(T\vert_{ \mathcal{M}})= \{ x\in \mathcal{D}(T)\cap \mathcal{M}:Tx \in \mathcal{M}\} ~~\mbox{and} ~~
T\vert_{ \mathcal{M}}~~x=Tx,~~x\in \mathcal{D}(T\vert_{ \mathcal{M}}).$$
 If $ \mathcal{M} $ reduce $ T $ to an operator $ B, $ then $ B= T\vert_{\mathcal{M}}$. Throughout the paper we present known result as proposition. In the following result we show that part of closed densely defined $ M $- hyponormal operator is again $ M $-hyponormal.
\begin {lemma}\label{restriction}
Let $ T \in \mathcal{C}(\mathcal{H}) $ be a densely defined $M$-hyponormal operator and let $ \mathcal{M} $ be a closed subspace of $\mathcal{H}$ which is invariant under T.
Then $ T\vert_\mathcal{M} $ is a closed $M$-hyponormal operator.
\end{lemma}

\begin{proof}

We have $ \mathcal{D}(T\vert_\mathcal{M}) = \mathcal{D}(T)\cap \mathcal{M}$. Let $ x \in \mathcal{D}(T\vert_{\mathcal{M}}) $ and $ P $ be an orthogonal projection on to $ \mathcal{M} $. Then
\begin{align*} \parallel (T\vert_\mathcal{M}-\lambda I)^{*}x \parallel &= \parallel P(T-\lambda I)^{*}x \parallel \\
&\leq M \parallel(T-\lambda I)x \parallel \\
& = M \parallel (T\vert_\mathcal{M}-\lambda I)x\parallel.
\end{align*}
Hence $ T\vert_\mathcal{M} $ is a closed $M$-hyponormal operator.
\end{proof}

The next result for closed hyponormal operators has been studied in (\cite{stochel}) by Stochel. 
\begin{lemma}\label{contraction}

Let $ T \in \mathcal{C}(\mathcal{H}) $ be a densely defined $M$-hyponormal operator. Then there exist a contraction $ C_{\lambda} \in \mathcal{B}(\mathcal{H}) $ such that $ \dfrac{1}{M}(T - \lambda) \subseteq (T - \lambda)^{*} C_\lambda$ for every $ \lambda \in \mathbb{C} $.

\end{lemma}

\begin{proof}
Define $ K :ran(T - \lambda)\rightarrow ran(T^* - \overline{\lambda}) $ by \begin{align*}
K \left((T - \lambda)x \right)= \dfrac{1}{M}(T^{*} - \overline{\lambda})x, ~~\mbox{for all}  ~~~x\in \mathcal{D}(T).
\end{align*}  Since $ T $ is $ M $-hyponormal, $ K $ is a contraction  with $ K(T - \lambda)\subseteq \dfrac{1}{M}(T^{*} - \overline{\lambda}).$ Now we extend $ K $ to $ K^{\prime} \in \mathcal{B}\left(\overline{ran(T - \lambda)},\overline{ran(T^* - \overline{\lambda})}\right)$ such that \\ $ K^{\prime}(T - \lambda)\subseteq \dfrac{1}{M}(T^{*} - \overline{\lambda})$. Then the contraction   $ A \in \mathcal{B}(\mathcal{H}) $ defined by $ Ax =0~~ $ for all $ x \in {\overline{ran(T - \lambda)}}^{\perp} $ is an extension of $ K^{\prime} $. Hence $$ A(T - \lambda)\subseteq \dfrac{1}{M}(T^{*} - \overline{\lambda}).$$ Therefore, we get $$ \dfrac{1}{M}(T - \lambda)\subseteq (T - \lambda)^{*}A^*.$$ If we put $ A^* = C_\lambda, $ then $ \dfrac{1}{M}(T - \lambda) \subseteq (T - \lambda)^{*} C_\lambda.$
\end{proof}
 A closed subspace $ \mathcal{M} $ of $ \mathcal{H} $ \textit{reduces} $ T \in \mathcal{C}(\mathcal{H}) $ if $ \mathcal{M} $ and $ \mathcal{M^{\perp}} $ are invariant under $ T $. Stochel (\cite{stochel}) proved  if $T \in \mathcal{C}(\mathcal{H})$ is hyponormal and $ \mathcal{M}$ is a closed subspace of $\mathcal{H}$ which is invariant under $ T $ with $ T\vert_\mathcal{M} $ is normal, then $ \mathcal{M} $ reduces $ T. $ Now we prove the result for closed densely defined $M$-hyponormal operators. 

\begin{theorem} \label{reduces}
Suppose $ T \in \mathcal{C}(\mathcal{H}) $ is a densely defined $M$-hyponormal operator. If $ \mathcal{M}$ is a closed subspace of $\mathcal{H}$ which is invariant under $ T $ with $ T\vert_\mathcal{M} $ is normal, then $ \mathcal{M} $ reduces $ T. $
\end{theorem}

\begin{proof}
Let $ \mathcal{H}=\mathcal{H}_{1} \oplus\mathcal{H}_{2},$ where $ \mathcal{H}_{1}=\mathcal{M} $ and $ \mathcal{H}_{2} = \mathcal{M}^{\perp}.$ Then $ T $ has the block matrix representation  $$ T = \begin{bmatrix}
T_{11}& T_{12}\\
T_{21}& T_{22}
\end{bmatrix},$$ where $ T_{ij}:D(T)\cap \mathcal{H}_{j}\rightarrow  \mathcal{H}_{i} $ is defined by $T_{ij}=P_{\mathcal{H}_{i}} T P_{\mathcal{H}_{j}}\vert_{D(T)\cap \mathcal{H}_{j}}  $ for $ j=1,2. $ Here, $ P_{\mathcal{H}_{i}} $ denotes the orthogonal projection onto $ \mathcal{H}_{i}. $ Since $ \mathcal{M} $ is  invariant under $ T $, we have \\$$ T = \begin{bmatrix}
T_{11}& T_{12}\\
0& T_{22}
\end{bmatrix}.$$ \\Let $ y\in D(T)\cap \mathcal{M}^{\perp}. $ By Lemma \ref{contraction}, we have $$ \dfrac{1}{M}(T - \lambda) \subseteq (T - \lambda)^{*} C_\lambda$$ for every $ \lambda \in \mathbb{C} $. Thus, $ ran(T-\lambda)\subseteq ran(T- \lambda)^{*} $ for every $ \lambda \in \mathbb{C}.$ Then there exist a densely defined operator $ B $ such that $ (T-\lambda)=(T-\lambda)^{*}B $ ( see \cite{doug}). Hence, $ T_{12}(y)  =(T_{11} -\lambda)^{*}u $ for some $ u \in \mathcal{M}.$ We can choose $ v $ such that $ (T_{11}-\lambda)^{*}u = (T_{11}-\lambda)v.$ Therefore, $ T_{12}(y) = (T_{11}-\lambda)v $ for every $ \lambda \in \mathbb{C}.$ Hence, $$ T_{12}(y)\in \bigcap \limits_{\lambda \in \mathbb{C}}ran(T_{11}-\lambda).$$ Thus, $ T_{12}(y) = 0$ (\cite{putnam}) and so $ T_{12} = 0$.
\end{proof}

\begin{proposition}(\cite{stochel})\label{core}
Let $\mathcal{M} $ be a core for  $ T \in \mathcal{C}(\mathcal{H}) $ and  $ A \in \mathcal{B}(\mathcal{H})$ be a selfadjoint operator with $ ker(A) =\{0\}. $ If $ AT \subseteq TA $, then $A(\mathcal{M})$ is a core for $ T $.
\end{proposition}
 Now we extend the Theorem 2.3 of (\cite{stochel}) from closed hyponormal to closed $M$-hyponormal operator by the similar argument as in the proof of (\cite[Theorem 2.3]{stochel}).
\begin{theorem}\label{specmeasure}

Let $ S \in \mathcal{C}(\mathcal{H}) $ be  normal. Let $ T \in \mathcal{C}(\mathcal{K}) $ be $M$-hyponormal. If $ A \in \mathcal{B}(\mathcal{H},\mathcal{K}) $ be such that $ AS \subseteq TA $. Then \\
(i) $ |A|S \subseteq S|A|$ \\
(ii)  If $ A\geq 0, ker(A)= \{0 \} $ and $ \mathcal{K} = \mathcal{H} $, then $ S = T $.

\end{theorem}

\begin{proof}
(i) Suppose $ T $ is $ M $-hyponormal. Then by Lemma \ref{contraction}, we have \begin{equation}\label{eqn1}
 \dfrac{1}{M}(T -\lambda) \subseteq  (T-\lambda)^{*} C_\lambda ~~ \mbox{for all} ~~\lambda \in \mathbb{C}.
\end{equation}
 Let $ E $ be the spectral measure of $ S $ and let $ \Omega $ be a compact subset of $ \mathbb{C}. $ Since $ E $ is regular, it is sufficient to prove that  $ |A|E(\Omega)=E(\Omega)|A| $ for every compact set $ \Omega $ of $\mathbb{C}.$
 Since $ S $ is normal, we have $ ran(E(\Omega)) $ reduces $ S-\lambda $ and $  ran(E(\Omega)) \subset \mathcal{D}( S -\lambda) $ for every $ \lambda \in \mathbb{C}$.\\ For $ \lambda \notin  \Omega, $ define the function $ \psi(\lambda)= \int_{\Omega} \dfrac{1}{( z -\lambda)} E(dz)x $\\ Then $ Ax= A( S -\lambda)\psi(\lambda),~\mbox{for} ~~\lambda \notin \Omega.$
Since $ AS \subseteq TA, $ we have $Ax=(T -\lambda)A\psi(\lambda),~\mbox{for} ~~\lambda \notin \Omega.$
Hence by equation (\ref{eqn1}), $ A^*Ax = A^*M (T -\lambda)^{*} C_\lambda A\psi(\lambda).$
Since $ A( S -\lambda)\subseteq (T -\lambda)A,$ for $ \lambda \in \mathbb{C},$
 $ A^*Ax =(S -\lambda)^*MA^*C_{\lambda}A\psi(\lambda),~\mbox{for} ~~\lambda \notin \Omega$ and hence $ A^{*} Ax \in \bigcap \limits_{z \in \mathbb{C} \setminus \Omega^{*}}ran(S^* - z),$ where $ \Omega^{*}=\{z:\overline{z}\in \Omega\}.$ Then $ E(\mathbb{C}\setminus \Omega)A^* Ax =0 $ (\cite[Theorem 2.2]{stochel}). Therefore, $ A^* Ax =E(\Omega)A^* Ax. $ Since $ x\in  ran(E(\Omega))$ is arbitrary, $$A^*A( ran(E(\Omega)))\subseteq  ran(E(\Omega)).$$ Since $ A^* A $ is selfadjoint, $ A^*A E(\Omega)= E(\Omega)A^*A. $ This completes the proof.\\

(ii) Since $ A\geq 0, $ from (i) we have $ SAx=ASx=TAx $ for $ x\in \mathcal{D}(S).$ Thus, $ S|_{A\mathcal{D}(S)} \subseteq T. $ Since $ \mathcal{D} $  is a core for $ S ,$ $ A\mathcal{D}(S) $ is a core for $ S $ from Proposition \ref{core}. Hence, $$ \mathcal{G}(S)\subseteq \overline {\mathcal{G}(S|_{A\mathcal{D}(S)})} \subseteq \overline{\mathcal{G}(T)} = \mathcal{G}(T).$$ Therefore, $ S \subseteq T $. Since $ D(T)\subseteq D(T^*) $, we have $ D(T)\subseteq D(T^*)\subseteq D(S^*)=D(S). $ Hence, $ S=T. $

\end{proof}

Let $ A \in \mathcal{B}(\mathcal{H},\mathcal{K}), $ we denotes $ \mathcal{R}(A^*):= \overline{ran(A^{*})}= \overline {ran( \vert A \vert)} $ and $\mathcal{R}(A) := \overline{ran(A)}.$ Let $ A = U ~~\vert A \vert $ be the polar decomposition of $ A, $ where $\vert A \vert = (A^{*}A)^{\frac{1}{2}}  $, $ U \in \mathcal{B}(\mathcal{H},\mathcal{K}) $ is partial isometry with initial space $\mathcal{R}(A^*)$ and final space $\mathcal{R}(A)$. Also $ ker(U)=ker(A) $ and $ U \vert_{\mathcal{R}(A^*)}  , A\vert_{\mathcal{R}(A^*)} $ are in  $\mathcal{B} (\mathcal{R}(A^*),\mathcal{R}(A))$. Also $ U \vert_{\mathcal{R}(A^*)} $ is a bounded unique unitary isomorphism from $ \mathcal{R}(A^*) $ into $\mathcal{R}(A)$ with $$ U\vert_{\mathcal{R}(A^*)}~~ |A|\,x~=~A\,x, x\in \mathcal{H}. $$



\begin{proposition}(\cite{stochel},Theorem 3.2)\label{stochel}

Let $ T $ and $ B $ be a closed densely defined operators in $\mathcal{H},$ and $ \mathcal{K}$ respectively, and let $ A\in \mathcal{B}(\mathcal{H},\mathcal{K}) $ be such that
$ AT^* \subseteq BA. $\\ (i) If $ \mathcal{R}(A^*) $ reduces $ T, $ then $ B|_{\mathcal{R}(A)} $ is closed densely defined operator in $ \mathcal{R}(A)$ and $$ A \vert_{\mathcal{R}(A^*)} ~~(T \vert_{\mathcal{R}(A^*)})^{*}\subseteq B\vert_{\mathcal{R}(A)}~~ A \vert_{\mathcal{R}(A^*)}.$$\newline
(ii) If $ \mathcal{R}(A^*) $ and $ \mathcal{R}(A) $ reduces $ T $ and $ B $ to normal operators respectively, then \\ $$ AT \subseteq B^*A, ~~ |A|T \subseteq T|A|,~~ |A^*|B \subseteq B |A^*|,$$ $$ (T \vert_{\mathcal{R}(A^*)})^{*}= (U\vert_{\mathcal{R}(A^*)})^{*} B\vert_{\mathcal{R}(A)} ~~U\vert_{\mathcal{R}(A^*)}.$$
\end{proposition}

Now we prove the asymmetric Fuglede-Putnam theorem for closed densely defined $ M $-hyponormal and normal operators. The following result is the motivated by (\cite{stochel},Proposition 4.1).
\begin{theorem} \label{main}
Suppose $ S \in \mathcal{C}(\mathcal{H}) $ is a normal operator and $  T\in \mathcal{C}(\mathcal{K}) $  is $ M $-hyponormal. If $ A \in \mathcal{B}(\mathcal{H},\mathcal{K}) $ is such that $ AS\subseteq TA. $ Then $ \mathcal{R}(A^*) $ reduces $ S $ and $T\vert_{\mathcal{R}(A)},~ S\vert_{\mathcal{R}(A^*)}$ are unitarily  equivalent normal operators.

\end{theorem}

\begin{proof}
Let  $ \Omega $ be a Borel subset of $\mathbb{C}$ and let $E $ be the spectral measure of $ S$. We have $ E(\Omega) $ is an orthogonal projection. To prove $ \mathcal{R}(A^*) $ reduces $ S, $ it is sufficient to prove that $\mathcal{R}(A^*)= \overline {ran( \vert A \vert)} $ is invariant under $ E(\Omega)$.

Let $ y \in \overline {ran( \vert A \vert)}$. Then there exist a sequence $ (y_n)\in ran( \vert A \vert) $ such that $ y_n$ converges to $y $. Since $ E(\Omega)$ is bounded, $ E(\Omega)y_n$ converges to $E(\Omega)y$ . Since $ (y_n)\in ran( \vert A \vert) $, there exist $ x_n \in \mathcal{D} (|A|)$ such that $ y_n=|A|x_n. $ Therefore, $ E(\Omega)|A|x_n$ converges to $ E(\Omega)y. $ From Theorem \ref{specmeasure} (i), $ |A|E(\Omega)=E(\Omega)|A|. $ Hence, $ |A|E(\Omega)x_n$ converges to $E(\Omega)y.$ Thus, $ \mathcal{R}(A^*) $ reduces $ S $.

Since $ AS\subseteq TA $ and $ \mathcal{R}(A^*) $ reduces $ S,$ we have $ T\vert_{\mathcal{R}(A)} $ is a closed densely defined operator in $ \mathcal{R}(A) $ and \begin{equation}\label{4}
 A \vert_{\mathcal{R}(A^*)} ~~ S \vert_{\mathcal{R}(A^*)}\subseteq T\vert_{\mathcal{R}(A)}~~ A \vert_{\mathcal{R}(A^*)}.
\end{equation} by Proposition \ref{stochel} (i). Since $ T $ is $ M $-hyponormal, $ T\vert_{\mathcal{R}(A)} $ is a closed $ M $-hyponormal operator in $ \mathcal{R}(A) $ by Lemma \ref{restriction}. From (\cite[Lemma 3.1(v)]{stochel}),
 we have $$ |~A \vert_{\mathcal{R}(A^*)}~| ~~~S \vert_{\mathcal{R}(A^*)} \subseteq (U \vert_{\mathcal{R}(A^*)})^* ~~T\vert_{\mathcal{R}(A)}~~U \vert_{\mathcal{R}(A^*)}~~ |A \vert_{\mathcal{R}(A^*)}| . $$
Let $ W = U \vert_{\mathcal{R}(A^*)}$, and  $V = W^* ~~T\vert_{\mathcal{R}(A)}~~W .$
Also we have $W$ is unitary isomorphism and $ T\vert_{\mathcal{R}(A)}$ is  $M$-hyponormal. Then for $ x \in \mathcal{R}(A^*),$ \\ 
\begin{align*}
\parallel (V- z I)x \parallel &= \parallel  W^* (~~T\vert_{\mathcal{R}(A)}~- z I)^*~~Wx \parallel \\ &= \parallel (~~T\vert_{\mathcal{R}(A)} ~-z I)^*~~Wx \parallel   \\&\leq M \parallel (~~T\vert_{\mathcal{R}(A)} ~-z I)Wx \parallel  \\ &= M \parallel W^*(~~T\vert_{\mathcal{R}(A)} ~-z I)Wx \parallel ~~~ \\ &= M \parallel (V-zI)x \parallel.
\end{align*}
Hence $(U \vert_{\mathcal{R}(A^*)})^* ~~T\vert_{\mathcal{R}(A)}~~U \vert_{\mathcal{R}(A^*)} $ is a closed $ M $-hyponormal operator.\\
 From Theorem \ref{specmeasure} (ii), we get  $ S \vert_{\mathcal{R}(A^*)}=(U_{A} \vert_{\mathcal{R}(A^*)})^* ~~T\vert_{\mathcal{R}(A)}~~U_{A} \vert_{\mathcal{R}(A^*)} $ because $ ker(|~A \vert_{\mathcal{R}(A^*)}~|)= ker(A \vert_{\mathcal{R}(A^*)})= \{0\}$. Thus, $T\vert_{\mathcal{R}(A)},~ S\vert_{\mathcal{R}(A^*)}$ are unitarily  equivalent normal operators.

\end{proof}
Stochel (\cite{stochel}) proved the following result for closed hyponormal and closed subnormal operators. Now we extend the result to closed $M$-hyponormal and closed subnormal operators using the method of (\cite[Theorem 4.2]{stochel})
\begin{theorem}
Let $ B\in \mathcal{C}(\mathcal{H})$ be subnormal (resp. a closed $ M $-hyponormal operator in $ \mathcal{H} $ ), $ T \in \mathcal{C}(\mathcal{K}) $ be $ M $-hyponormal (resp. a closed subnormal operator in $ \mathcal{K} $ ) and $ A \in \mathcal{B}(\mathcal{H},\mathcal{K})$ is such that $ AB^*\subseteq TA. $ Then \\ (i) $ AB\subseteq T^*A .$\\(ii) $ \mathcal{R}(A^*) $ reduces $ B $ to the normal operator $ B\vert_{\mathcal{R}(A^*)}. $\\(iii) $ \mathcal{R}(A) $ reduces $ T $ to the normal operator $ T\vert_{\mathcal{R}(A)} .$
\end{theorem}

\begin{proof}
First assume that $ B\in \mathcal{C}(\mathcal{H})$ is subnormal and $ T \in \mathcal{C}(\mathcal{K}) $ is $ M $-hyponormal. Since $ B $ is subnormal, there exist a normal extension $ S $ on the Hilbert space $ \mathcal{L} \supseteq \mathcal{H}. $ Define $ Y \in \mathcal{B}(\mathcal{K},\mathcal{L})$ by $ Yx=A^*x ,x\in \mathcal{K}.$ From (\cite[Theorem 3.8]{stochel}), we have $ \mathcal{R}(Y^*)= \mathcal{R}(A). $ \\ Since $ AB^*\subseteq TA, $ we have  \begin{equation} \label{sto}
A^* T^*\subseteq BA^*.
\end{equation}
Since from equation (\ref{sto}) and $ Yx=A^*x, $ we have 
\begin{align*}
YT^*x &= A^*T^*x \\ &= BA^*x \\&= SA^*x \\&=SYx, ~x\in \mathcal{D}(T^*). 
\end{align*}
Thus $ YT^*\subseteq SY. $ Hence $ Y^*S^*\subseteq TY^*. $ Since $ \mathcal{R}(Y^*)= \mathcal{R}(A) $,    we have $ T\vert_{\mathcal{R}(A)} $ and $ S\vert_{\mathcal{R}(A^*)}$ are unitarily equivalent normal operators by Theorem \ref{main}. Thus, $\mathcal{R}(A)  $ reduces $ T $ to the normal operator $ T\vert_{\mathcal{R}(A)} $. \\  Since $ A^* T^*\subseteq BA^*~\mbox{and} ~\mathcal{R}(A) $ reduces $ T  $,  we have
$ B\vert_{\mathcal{R}(A^*)}$ is closed densely defined in $ \mathcal{R}(A^*) $ and
\begin{equation}\label{2} 
(A \vert_{\mathcal{R}(A^*)})^* ~~(T \vert_{\mathcal{R}(A)})^{*}\subseteq B\vert_{\mathcal{R}(A^*)}~~ (A \vert_{\mathcal{R}(A^*)})^*
\end{equation} by Proposition \ref{stochel} (i) and (\cite[Lemma 3.1(iii)]{stochel}). Since $  B\vert_{\mathcal{R}(A^*)} \subseteq B \subseteq S, $  $  B\vert_{\mathcal{R}(A^*)} $ is subnormal. Then $  B\vert_{\mathcal{R}(A^*)} $ is $ M $-hyponormal. Thus, $ \mathcal{R}(A \vert_{\mathcal{R}(A^*)})^*=\mathcal{R}(A^*) $ reduces $ B\vert_{\mathcal{R}(A^*)} $ to the normal operator from equation (\ref{2}) and  Theorem \ref{main}. Hence by Theorem \ref{reduces}, $ \mathcal{R}(A^*) $ reduces $ B $ to the normal operator $ B\vert_{\mathcal{R}(A^*)}.$  The result (i) follows from Proposition \ref{stochel} (ii).\\
Next we assume that  $ B\in \mathcal{C}(\mathcal{H})$ is $ M $-hyponormal and $ T \in \mathcal{C}(\mathcal{K}) $ is subnormal. Since $ A^*T^* \subseteq BA^*,$ $ \mathcal{R}(A) $ and  $ \mathcal{R}(A^*) $ are reducing subspace for $ T $ and $ B $ respectively. Also the part $ T\vert_{\mathcal{R}(A)} $ and  $ B\vert_{\mathcal{R}(A^*)}$ are normal. By Proposition \ref{stochel} (ii), we have $ AB\subseteq T^*A.$ 

\end{proof}
\begin{corollary}
Let $ B\in \mathcal{C}(\mathcal{H})$ be subnormal (resp. a closed $ M $-hyponormal operator in $ \mathcal{H} $ ), $ T \in \mathcal{C}(\mathcal{K}) $ be $ M $-hyponormal (resp. a closed subnormal operator in $ \mathcal{K} $ ) and $ A \in \mathcal{B}(\mathcal{H},\mathcal{K})$ is such that $ AB^*\subseteq TA. $ Then \\ (i) If $ ker(A)=\{0\}, $ then $ B $ is normal.\\(ii) If $ ker(A^*)=\{0\}, $ then $ T $ is normal.
\end{corollary}

\end{document}